\documentclass[reqno,11pt]{amsart}

\usepackage{tikz}
\usetikzlibrary{arrows.meta,positioning}
\usepackage[utf8]{inputenc}
\usepackage[english]{babel}
\usepackage{enumerate}
\usepackage{graphicx}
\usepackage{xcolor}
\usepackage[colorlinks=true,linkcolor=blue,citecolor=green!50!black,urlcolor=purple]{hyperref}
\usepackage[left=2.5cm,right=2.5cm,bottom=2cm]{geometry}
\usepackage{amsmath,amssymb,amsthm,mathtools}
\usepackage{mathrsfs}
\usepackage[nameinlink,noabbrev]{cleveref}
\usepackage{comment}

\newcommand{\R}{\mathbb{R}}
\newcommand{\N}{\mathbb{N}}
\newcommand{\T}{\mathbb{T}}
\renewcommand{\S}{\mathbb{S}}
\newcommand{\Z}{\mathbb{Z}}
\newcommand{\C}{\mathbb{C}}

\newcommand{\id}{\mathrm{Id}}

\renewcommand{\Re}{\operatorname{Re}}
\renewcommand{\Im}{\operatorname{Im}}

\newcommand{\PWC}{\operatorname{PWC}}

\DeclareMathOperator{\Ad}{Ad}
\DeclareMathOperator{\Diff}{Diff}
\DeclareMathOperator{\Lie}{Lie}
\DeclareMathOperator{\tr}{tr}

\newtheorem{theorem}{Theorem}[section]
\newtheorem{lemma}{Lemma}[section]
\newtheorem{proposition}{Proposition}[section]
\newtheorem{corollary}{Corollary}[section]

\theoremstyle{definition}
\newtheorem{definition}{Definition}[section]

\theoremstyle{remark}
\newtheorem{remark}{Remark}[section]

\crefname{theorem}{Theorem}{Theorems}
\crefname{lemma}{Lemma}{Lemmas}
\crefname{proposition}{Proposition}{Propositions}
\crefname{corollary}{Corollary}{Corollaries}
\crefname{assumption}{Assumption}{Assumptions}
\crefname{definition}{Definition}{Definitions}
\crefname{example}{Example}{Examples}
\crefname{remark}{Remark}{Remarks}

\begin{document}

\title[Small-time controllability on homogeneous spaces]
{Small-time approximate controllability of bilinear Schr\"odinger
equations on compact normal homogeneous spaces}

\author[H. Parada]{ Hugo Parada}
\address{Universit\'e de Lorraine, CNRS, Inria, IECL,
F-54000 Nancy, France}
\email{hugo.parada@inria.fr}
\thanks{This project is funded by the Agence Nationale de la Recherche through the QuBiCCS project (ANR-24-CE40-3008)}

\begin{abstract}
We prove small-time global approximate controllability in $L^2$ for bilinear Schr\"odinger equations on compact connected normal homogeneous spaces. The controls act through finitely many smooth real-valued potentials. We first construct finite saturating families on compact connected Lie groups endowed with bi-invariant metrics, using Peter--Weyl theory and the Casimir operator. Their phase saturation contains the representative algebra and is therefore dense in $L^2$. We extend this construction to normal homogeneous quotients using finite generation of the algebra of highest-weight functions. Global controllability follows by combining the phase-to-gradient transport argument of \cite{BP25} with the gradient generation theorem of \cite{KMP26} and Moser's transport of positive densities.
\end{abstract}

\maketitle

\section{Introduction and main results}
\label{sec:intro}
Let $(M,g)$ be a smooth, connected, compact and boundaryless Riemannian manifold. We consider the bilinear Schr\"odinger equation
\begin{equation}
\label{eq:Sch}
\begin{cases}
i\partial_t\psi(t)=\displaystyle\left(-\Delta_g+V+\sum_{j=1}^m u_j(t)W_j\right)\psi(t),&t\in(0,T),\\
\psi(0)=\psi_0\in L^2(M,\C),
\end{cases}
\end{equation}
where $V\in L^\infty(M,\R)$, $W_1,\ldots,W_m\in C^\infty(M,\R)$, and $u=(u_1,\ldots,u_m)\in\PWC([0,T];\R^m)$ is piecewise constant. The solution $\psi(t;\psi_0,u)$ is well-defined and preserves the $L^2$ norm. Thus, the natural state space is
\[\S_M:=\{\psi\in L^2(M,\C):\|\psi\|_{L^2}=1\}.\]
We study the control of Schr\"odinger equations through finitely many spatial potentials. Our aim is to approximate arbitrary states in arbitrarily small time by exploiting the geometry of compact Lie groups and normal homogeneous spaces.

\subsection{Literature review}
Approximate controllability is a natural question in the framework of bilinear systems in infinite dimension. Indeed, the non-controllability theorem of \cite[Theorem~3.6]{BMS82} implies that bounded multiplicative control operators and $L^p$ controls, $p>1$, cannot yield exact controllability on the infinite-dimensional $L^2$ unit sphere. Positive results for approximate controllability for discrete-spectrum Schr\"odinger equations were obtained in \cite{CMSB09} under spectral and coupling conditions. Inspired by the geometric control method of Agrachev and Sarychev for the Navier--Stokes and Euler equations with finite-dimensional additive controls~\cite{AS05,AS06}, a multiplicative saturation method was developed for nonlinear Schr\"odinger equations on $\T^d$; see \cite[Theorem~2.2]{DN25}. This method generates phase multipliers from the controlled potentials through a saturation procedure. In simple terms, the existence of a saturating family, together with suitable PDE estimates, ensures small-time approximate controllability of phases. This sufficient condition was extended to linear Schr\"odinger equations on closed Riemannian manifolds and an explicit finite saturating family on $S^2$ was constructed; see \cite[Theorems~2, 3 and~6]{CP23}. Related developments include bilinear electromagnetic controls on $\R^N$~\cite{DP25}, as well as manifolds with boundary and quantum graphs~\cite{BLS25}. In the compact setting, the explicit examples in these references include the torus $\T^d$, the sphere $S^2$ and a figure-eight quantum graph. Phase multiplication preserves the pointwise modulus of a wave function. In \cite{BP25} global controllability was obtained by combining phase control with transport by diffeomorphisms, building on Thurston's simplicity theorem for diffeomorphism groups and Moser's theorem on volume forms. \medskip

\noindent The construction problem addressed here is therefore the following: \emph{Can one choose finitely many smooth potentials whose phase saturation is dense on a given compact manifold?} This question in a general setting was catalogued as an ambitious problem in \cite{BP25}. We give an answer to this question for compact connected Lie groups with bi-invariant metrics and for their normal homogeneous quotients. Our construction uses matrix coefficients, the Casimir operator and highest-weight representations. It produces a saturation space containing the full representative algebra and yields global $L^2$ approximate controllability in arbitrarily small time. This representation-theoretic construction also provides a step towards a bilinear counterpart to \cite{BDP26}, where small-time approximate controllability for nonlinear PDEs on compact manifolds was proved in the case of finite-dimensional additive control. In a related work \cite{BDP26-2}, we actually solve this problem for a general compact Riemannian manifold, with an argument based on Nash's embedding. Here, we use fewer control profiles and, in some sense, more explicit ones\footnote{For instance, on the round sphere $S^d$ and on $SO(3)$ we use $d+2$ and ten potentials, respectively (seven suffice on $SO(3)$), while the construction of \cite{BDP26-2} gives upper bounds of $3^{d+1}$ and $3^9=19683$, using isometric embeddings into $\R^{d+1}$ and $\R^9$, respectively.}.\\

\noindent More precisely, here we let $G$ be a compact connected Lie group with a bi-invariant Riemannian metric and let $K\subset G$ be a closed subgroup. We equip $M=G/K$ with the quotient metric for which the projection $G\to G/K$ is a Riemannian submersion; such a metric is called \emph{normal homogeneous}. This setting includes compact Lie groups themselves, by taking $K=\{1\}$, as well as spheres, projective spaces and Grassmannians with their standard normal homogeneous metrics. Lie groups and homogeneous spaces arise naturally in Schr\"odinger dynamics and quantum control. For instance, the rotational dynamics of a rigid linear molecule can be described by a Schr\"odinger equation on $S^2\simeq SO(3)/SO(2)$~\cite{CP23}, while $SO(3)$ is the configuration space of a rigid rotor~\cite{KG22}. Bilinear control of molecular rotations on $SO(3)$ has been studied for symmetric and asymmetric molecules in~\cite{BPS21,Poz22}, respectively. We use 
\[SO(3):=\left\{R\in M_3(\mathbb R):R^\top R=I_3,\ \det R=1\right\},\]
as a running example of the group construction. Unit quaternions describe rotor orientations through the universal covering $\Pi:SU(2)\simeq\operatorname{Spin}(3)\to SO(3)$: $q$ and $-q$ represent the same orientation. A scalar wave function $\psi$ therefore lifts to $\widetilde\psi=\psi\circ\Pi$, with $\widetilde\psi(-q)=\widetilde\psi(q)$ almost everywhere. This explains why only integer angular momentum representations occur in our $SO(3)$ example; see \Cref{rem:quaternion-cover} for the quaternionic realization. Group actions and representations have also been exploited in geometric
control theory: generic controllability of systems with symmetries is studied
in \cite{AL26}, and in \cite{AgrachevKazandjianPozzoli2025} a unitary representation of $SL_2(\mathbb{R})\ltimes H_d(\mathbb{R})$ is used to obtain
small-time reachability results for the quantum harmonic oscillator. There the group acts as a control group on $\mathbb{R}^d$, while in our case $G/K$ is the configuration space itself.

\subsection{Main results}

\begin{definition}[$L^2$-STAC]
We say that \eqref{eq:Sch} is small-time globally approximately controllable in $L^2$, or $L^2$-STAC, if, for every $\psi_0,\psi_1\in\S_M$ and $\varepsilon>0$, there exist $T\in[0,\varepsilon]$ and $u\in\PWC([0,T];\R^m)$ such that
\[\|\psi(T;\psi_0,u)-\psi_1\|_{L^2}<\varepsilon.\]
\end{definition}

The first step is to control spatially varying phases. Following \cite[Definitions~2, 6 and~12]{BP25}, we formulate this controllability up to a constant global phase.

\begin{definition}[STAC of phases]
We say that \eqref{eq:Sch} satisfies the small-time approximate controllability of phases if, for every $\psi_0\in\S_M$, $\varphi\in L^2(M,\R)$ and $\varepsilon>0$, there exist $T\in[0,\varepsilon]$, $\theta\in[0,2\pi)$ and $u\in\PWC([0,T];\R^m)$ such that
\begin{equation}
\label{eq:stac-phase}
\bigl\|\psi(T;\psi_0,u)-e^{i\theta}e^{i\varphi}\psi_0\bigr\|_{L^2}<\varepsilon.
\end{equation}
\end{definition}
The main results of this work are the following:
\begin{theorem}[Phase control on compact Lie groups]
\label{thm:main-phases-G}
Let $M=G$ be a compact connected Lie group endowed with a bi-invariant Riemannian metric. Let $\mathcal W_G$ be the finite-dimensional space defined in \eqref{eq:control-space-G}. For every $V\in L^\infty(G,\R)$, equation \eqref{eq:Sch}, with controlled potentials given by a real basis of $\mathcal W_G$, is small-time approximately controllable in phase.
\end{theorem}

\begin{theorem}[Phase control on normal homogeneous spaces]
\label{thm:main-phases-homogeneous}
Let $M=G/K$ be a compact, connected normal homogeneous space as above. There exists a finite-dimensional $G$-invariant space $\mathcal W_M\subset C^\infty(M,\R)$, containing $1$, such that, for every $V\in L^\infty(M,\R)$, equation \eqref{eq:Sch}, with controlled potentials given by a real basis of $\mathcal W_M$, is small-time approximately controllable in phase.
\end{theorem}

The two constructions are given in \Cref{sec:finite-saturation-groups,sec:finite-saturation-homogeneous}. The passage to global control is expressed by the following result, which does not require a homogeneous structure.

\begin{proposition}[Global control]
\label{prop:main-global}
Let $(M,g)$ be a smooth, connected, closed Riemannian manifold, let $V\in L^\infty(M,\R)$, and let $W_1,\ldots,W_m\in C^\infty(M,\R)$. Assume that $1\in\operatorname{span}_{\R}\{W_1,\ldots,W_m\}$. If \eqref{eq:Sch} is small-time approximately controllable in phase, then it is $L^2$-STAC.
\end{proposition}

The main contribution of this paper is the construction of finite saturating families on compact connected Lie groups with bi-invariant metrics and on normal homogeneous spaces. The passage from phase control to global controllability in \Cref{prop:main-global} combines the gradient transport argument of \cite{BP25} with the gradient generation theorem of \cite[Theorem~A]{KMP26}. We give the proof on closed Riemannian manifolds in \Cref{sec:global-control-short}.

\begin{corollary}[Global control with finitely many potentials]
\label{cor:main-global}
Under the assumptions of \Cref{thm:main-phases-G}, respectively \Cref{thm:main-phases-homogeneous}, equation \eqref{eq:Sch} is $L^2$-STAC for every $V\in L^\infty(M,\R)$ when the controlled potentials form a real basis of $\mathcal W_G$, respectively $\mathcal W_M$.
\end{corollary}

\subsection{Heuristic ideas}

The saturation mechanism is governed by the following quadratic operations: 
\begin{equation}
\label{eq:Gamma-intro}
\Gamma(f,h):=\langle\nabla f,\nabla h\rangle_g, \qquad B(f):=\Gamma(f,f)=|\nabla f|_g^2, \qquad f,h\in C^\infty(M,\R).
\end{equation}
Short conjugations by large phases generate effective potentials of the form $-B(f)$; see \cite[Theorem~2]{CP23} and \Cref{subsec:phase-saturation}. A key point is to recover both signs of these generated potentials. The basic observation is that symmetry provides finite families with constant total energy:
\[\sum_{j=1}^m B(W_j)\equiv C.\]
For example, trigonometric functions on the torus and coordinate functions on the unit sphere satisfy
\[B(\sin\langle k,x\rangle)+B(\cos\langle k,x\rangle)=|k|^2 \text{ on } \T^d, \qquad \sum_{j=1}^{d+1}B(x_j)=d \text{ on } S^{d}.\]
More generally, if $-\Delta_g W_j=\lambda W_j$ and $\sum_j W_j^2\equiv C_0$, the product rule gives $\sum_j B(W_j)=\lambda C_0$. Such identities allow the recovery of positive quadratic terms from negative ones and constant functions. Polarization then gives access to the bilinear operation $\Gamma$. On a compact Lie group, the Casimir eigenvalue formula identifies interactions under $\Gamma$ between matrix coefficients. The Cartan component of a tensor product yields a propagation rule for highest weights, from which we construct a finite family generating the full representative algebra. On a normal homogeneous quotient, we use finitely many generators of the algebra of highest-weight functions. Peter--Weyl theory then gives the required density in $L^2$. Finally, phase control yields the unitary transport operators associated with gradient flows. Lie brackets give all smooth flows by \cite[Theorem~A]{KMP26}. Approximating smooth isotopies by autonomous flows then gives control of the unitary action of $\Diff_0(M)$. Moser's transport of positive densities, followed by phase correction and an $L^2$ density argument, gives global $L^2$-STAC. A road map of the results is shown in \Cref{fig:phase-to-global}. We recall that our main objective is to \emph{construct finite saturating families on compact Lie groups and normal homogeneous spaces}.
\begin{figure}[htbp]
\centering
\begin{tikzpicture}[
    node distance=9mm and 7mm,
    box/.style={
        draw,
        rounded corners=1mm,
        align=center,
        text width=3.8cm,
        minimum height=18.5mm,
        inner sep=2.5pt,
        font=\small
    },
    arrow/.style={
        -{Latex[length=2mm]},
        line width=.6pt
    }
]

\node[box] (schrodinger)
    {Schr\"odinger equation on\\
     $M=G$ or $M=G/K$};

\node[box,right=of schrodinger] (phases)
    {STAC of phases\\[2pt]
     {\footnotesize
      \Cref{thm:main-phases-G} ($G$),\\
      \Cref{thm:main-phases-homogeneous} ($G/K$)}};

\node[box,right=of phases] (gradients)
    {STAC of gradient flows\\[2pt]
     {\footnotesize
      \Cref{thm:gradient-flows-short}\\
      adapted from\\
      \cite[Theorem~16]{BP25}}};

\node[box,below=of gradients] (flows)
    {STAC of flows\\[2pt]
     {\footnotesize
      \cite[Theorem~A]{KMP26}}};

\node[box,left=of flows] (diffeos)
    {STAC of $\operatorname{Diff}_0(M)$\\[2pt]
     {\footnotesize
      \Cref{cor:flows-diffeomorphisms-short}\\
      adapted from\\
      \cite[Theorem~13]{BP25}}};

\node[box,left=of diffeos] (global)
    {Global $L^2$-STAC\\[2pt]
     {\footnotesize
      \Cref{prop:main-global}\\
      adapted from\\
      \cite[Theorem~12]{BP25}}};

\draw[arrow] (schrodinger) -- (phases);
\draw[arrow] (phases) -- (gradients);
\draw[arrow] (gradients) -- (flows);
\draw[arrow] (flows) -- (diffeos);
\draw[arrow] (diffeos) -- (global);

\draw[arrow] (phases.south) -- ++(0,-6mm)
    -| (global.north);

\end{tikzpicture}
\caption{From phase control to global approximate controllability
in arbitrarily small time.}
\label{fig:phase-to-global}
\end{figure}
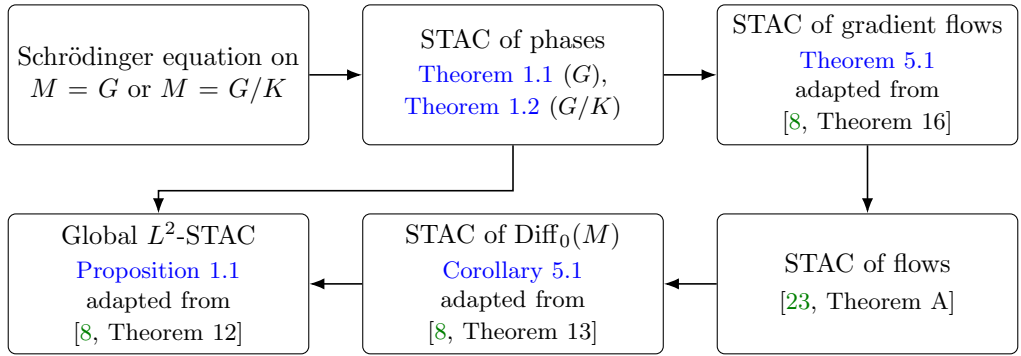

\section{Preliminaries}
\label{sec:preliminaries}
We first record the analytic framework for \eqref{eq:Sch} on a smooth, connected, closed Riemannian manifold $(M,g)$. By Stone's theorem (see \cite{Kato95}), we have the following well-posedness result.

\begin{proposition}
\label{prop:well-posedness}
For every $T>0$, $u\in\PWC([0,T];\R^m)$ and
$\psi_0\in L^2(M,\C)$, equation \eqref{eq:Sch} admits a unique mild
solution $\psi\in C([0,T];L^2(M,\C))$. For every $t\in[0,T]$, the map
$\psi_0\mapsto\psi(t;\psi_0,u)$ is unitary. In particular,
\[\|\psi(t;\psi_0,u)\|_{L^2}=\|\psi_0\|_{L^2},\qquad t\in[0,T].\]
Moreover, the solution depends continuously on $(\psi_0,u)$ in the
$L^2(M,\C)\times L^1(0,T;\R^m)$ topology, with values in
$C([0,T];L^2(M,\C))$.
\end{proposition}

\subsection{Phase saturation}
\label{subsec:phase-saturation}
We use the quadratic operation $B$ defined in \eqref{eq:Gamma-intro}.
For $\psi_0\in\S_M$, write
$\psi_1\in\overline{\operatorname{Reach}}_{\mathrm{st}}(\psi_0)$ if,
for every $\varepsilon,\tau>0$, there exist $T\in[0,\tau]$ and
$u\in\PWC([0,T];\R^m)$ such that
\[\|\psi(T;\psi_0,u)-\psi_1\|_{L^2}<\varepsilon.\]
Let $\mathcal W:=\operatorname{span}_{\R}\{W_1,\ldots,W_m\}$ and set
$\mathcal H_1:=\mathcal W$. Define recursively
\[\mathcal H_{n+1}:=\mathcal H_n+
\left\{-\alpha B(f):\alpha\geq0,\quad f,-f\in\mathcal H_n\right\},
\qquad n\geq1, \qquad \mathcal H_\infty(\mathcal W):=\bigcup_{n\geq1}\mathcal H_n.\]
We also write $\mathcal H_\infty$ when $\mathcal W$ is understood.

\begin{theorem}[Phase saturation, {\cite[Theorem~A.3]{BLS25}\footnote{Theorem~A.3 of the revised arXiv version}, \cite[Theorem~6]{BP25}}]
\label{thm:CP-phase}
For every $\psi_0\in\S_M$,
\[\left\{e^{i\varphi}\psi_0:\varphi\in\mathcal H_\infty\right\}
\subset\overline{\operatorname{Reach}}_{\mathrm{st}}(\psi_0).\]
Consequently, if $\mathcal H_\infty$ is dense in $L^2(M,\R)$, then \eqref{eq:Sch} is small-time approximately controllable in phase.
\end{theorem}

\subsection{Lie groups}
\label{sec:Lie_groups}
Throughout, $G$ is a compact connected Lie group. Let $1$ be its identity, let $\mathfrak g:=T_1G$ be its Lie algebra, and let $\mu$ be normalized Haar measure. On a compact group this measure is invariant under both left and right translations; see \cite[Corollary~2.28]{Folland16}. Every compact Lie group admits a bi-invariant Riemannian metric. One obtains such a metric by averaging an inner product on $\mathfrak g$ under the adjoint action and then transporting it by left translations; see \cite[Proposition~2.24]{AB15}. We fix a bi-invariant metric $\overline m$. Its normalized Riemannian volume measure coincides with $\mu$. We write $L^2(G):=L^2(G,\mu;\C)$ and denote by $\Delta_G:=\Delta_{\overline m}$ the associated Laplace--Beltrami operator. \smallskip

\noindent For $SO(3)$, one has $T_1G=\mathfrak{so}(3)$, where $\mathfrak{so}(3)=\{X\in M_3(\mathbb R):X^\top=-X\}$. The standard inner product $-\frac12\tr(XY)$ on $\mathfrak{so}(3)$ is $\Ad(SO(3))$-invariant and defines a bi-invariant metric on $SO(3)$. Since $\mathfrak{so}(3)$ is simple, this metric is unique up to rescaling by \cite[Proposition~2.48]{AB15}. Therefore, for some $c>0$, the bi-invariant metric is
\[\overline m_R(U,V)=-\frac{c}{2}\tr\big((R^{-1}U)(R^{-1}V)\big),\qquad R\in SO(3),\quad U,V\in T_RSO(3).\]

\subsection{Peter--Weyl theory}

We denote by $\widehat G$ the set of equivalence classes of irreducible unitary representations of $G$. These representations are finite-dimensional. For $\pi:G\to U(V_\pi)$, choose an orthonormal basis $e_1,\ldots,e_{d_\pi}$ of $V_\pi$, where $d_\pi:=\dim_{\C}V_\pi$. For $g\in G$, let $A(g)$ be the matrix of the operator $\pi(g)$ in this basis, so that
\[\pi(g)e_j=\sum_{i=1}^{d_\pi}A_{ij}(g)e_i.\]
Its coefficient space is the space of functions generated by these matrix entries:
\[\mathcal E_\pi:=\operatorname{span}_{\C}\{g\mapsto A_{ij}(g):1\leq i,j\leq d_\pi\}.\]
One may also write the coefficients as $g\mapsto\langle\pi(g)v,w\rangle_{V_\pi}$. We use the following form of the Peter--Weyl theorem; see \cite[Chapter~5]{Folland16}.

\begin{theorem}[Peter--Weyl, {\cite[Theorem~5.12]{Folland16}}]
\label{thm:peter-weyl}
The coefficient spaces are mutually orthogonal in $L^2(G)$, and
\[L^2(G)=\widehat{\bigoplus}_{\pi\in\widehat G}\mathcal E_\pi.\]
In particular, finite linear combinations of matrix coefficients are dense in $L^2(G)$.
\end{theorem}

The representative algebra is the algebraic direct sum
\[\mathcal R(G):=\bigoplus_{\pi\in\widehat G}^{\mathrm{alg}}\mathcal E_\pi.\]
Products of matrix coefficients are coefficients of tensor products, so $\mathcal R(G)$ is an algebra. It is stable under complex conjugation, and its real part $\mathcal R(G,\R):=\mathcal R(G)\cap C^\infty(G,\R)$ is dense in $L^2(G,\R)$. Left and right translations are isometries of the bi-invariant metric, and the action
\[((a,b)\cdot f)(g):=f(a^{-1}gb),\qquad a,b,g\in G,\]
preserves each $\mathcal E_\pi$. For irreducible $\pi$, the coefficient map $\xi\otimes v\mapsto[g\mapsto\xi(\pi(g)v)]$ is an isomorphism from $V_\pi^*\otimes V_\pi$ onto $\mathcal E_\pi$ by Schur orthogonality \cite[Theorem~5.8(b)]{Folland16}. It identifies this action with the outer tensor product $\pi^*\boxtimes\pi$. Since $\pi^*$ and $\pi$ are irreducible, $\mathcal E_\pi$ is irreducible under $G\times G$ by \cite[Theorem~7.12]{Folland16}. Thus the translates of any nonzero coefficient span $\mathcal E_\pi$. A \emph{continuous character} of a torus $T$ is a continuous group homomorphism $\chi:T\to\S^1$, equivalently a one-dimensional unitary representation of $T$.\smallskip

\noindent For $SO(3)$, the standard representation is $\pi_1:SO(3)\to U(\C^3)$, $\pi_1(R)v=Rv$. It is irreducible, and its matrix coefficients in the standard basis are $\langle\pi_1(R)e_j,e_i\rangle=R_{ij}$. Hence $\mathcal E_{\pi_1}=\operatorname{span}_{\C} \{R\mapsto R_{ij}:1\leq i,j\leq3\}$. The trivial representation $\pi_0$ has coefficient space $\mathcal E_{\pi_0}=\C1$. More generally, the irreducible representations are indexed by $\ell\in\N\cup\{0\}$, with $\dim_\C V_{\pi_\ell}=2\ell+1$. On $\T^d=\R^d/(2\pi\Z)^d$, write $g=[x]$ with $x\in\R^d$. All irreducible representations are the characters $\chi_k(g)=e^{i\langle k,x\rangle}$, $k\in\Z^d$, and Peter--Weyl reduces to Fourier decomposition.

\subsection{Spectral and product properties of matrix coefficients}
\label{subsec:coefficient-properties}
Let $\Lambda_G^+$ be the additive monoid of highest weights parametrizing the irreducible unitary representations $\pi_\alpha$ of $G$; see \cite[Theorem~5.110]{Knapp02}. A finite set of generators is constructed explicitly in \Cref{app:explicit-generators}. Write $V_\alpha:=V_{\pi_\alpha}$ and $\mathcal E_\alpha:=\mathcal E_{\pi_\alpha}$, with $\mathcal E_0=\C1$, and let $\langle\cdot,\cdot\rangle_*$ be the inner product on weights induced by the bi-invariant metric. The Casimir formula \cite[Proposition~5.28(b)]{Knapp02}, with the metric normalization of \Cref{app:Casimir-details}, gives $c_\alpha\geq0$ such that
\begin{equation}\label{eq:coefficient-eigenvalue}-\Delta_Gf=c_\alpha f,\qquad f\in\mathcal E_\alpha.\end{equation}
\begin{equation}\label{eq:parameter-gap}\frac{c_\alpha+c_\beta-c_{\alpha+\beta}}2=-\langle\alpha,\beta\rangle_*,\qquad\alpha,\beta\in\Lambda_G^+.\end{equation}
The Cartan component of a tensor product \cite[p.~309]{Baur03} provides coefficients $q_\alpha\in\mathcal E_\alpha$, normalized by $q_\alpha(1)=1$, such that
\[q_\alpha q_\beta\in\mathcal E_{\alpha+\beta}\setminus\{0\},\qquad\alpha,\beta\in\Lambda_G^+.\]
The definitions and constructions are detailed in \Cref{app:representation-details}. Extending $\Gamma$ complex-bilinearly, the Laplacian product rule and \eqref{eq:parameter-gap} yield
\begin{equation}\label{eq:Cartan-interaction}\Gamma(q_\alpha,q_\beta)=-\langle\alpha,\beta\rangle_*q_\alpha q_\beta.\end{equation}

\begin{lemma}[Propagation of a Cartan component]
\label{lem:Cartan-propagation}
Let $\mathcal A\subset C^\infty(G,\C)$ be a complex vector space invariant under $G\times G$ and stable under $\Gamma$. If $\mathcal E_\alpha,\mathcal E_\beta\subset\mathcal A$, $ \langle\alpha,\beta\rangle_*\neq0$, then $\mathcal E_{\alpha+\beta}\subset\mathcal A$.
\end{lemma}

\begin{proof}
By \eqref{eq:Cartan-interaction}, the nonzero function $q_\alpha q_\beta$ belongs to $\mathcal A\cap\mathcal E_{\alpha+\beta}$. Its $G\times G$-translates span $\mathcal E_{\alpha+\beta}$, so this entire space belongs to $\mathcal A$.
\end{proof}

\subsection{Normal homogeneous spaces}
\label{subsec:homogeneous-preliminaries}
Let $G$ be a compact connected Lie group endowed with a bi-invariant metric $\overline m$, and let $K\subset G$ be a closed subgroup. The quotient $M:=G/K$ is a compact connected smooth manifold, on which $G$ acts transitively by $a\cdot(gK):=agK$. Write $p:G\to M$ for the quotient projection. Let $\mathfrak k$ be the Lie algebra of $K$, and take its orthogonal complement with respect to $\overline m_1$. Write $m$ for the normal homogeneous metric on $M$, determined by
\[m_{1K}(dp_1X,dp_1Y)=\overline m_1(X,Y),\qquad X,Y\in\mathfrak k^\perp,\]
and $G$-invariance. Thus $p$ is a Riemannian submersion. The normalized Riemannian measure on $M$ is $\mu_M=p_*\mu_G$, where $\mu_G$ is normalized Haar measure. Pullback identifies functions on $M$ with right $K$-invariant functions on $G$:
\[(p^*f)(g)=f(gK),\qquad(p^*f)(gk)=(p^*f)(g),\qquad g\in G,\quad k\in K.\]
The fibers of $p$ are totally geodesic. Consequently, for $f,h\in C^\infty(M,\C)$,
\begin{equation}
\label{eq:homogeneous-pullback}
\Delta_G(p^*f)=p^*(\Delta_mf),
\qquad
\Gamma_G(p^*f,p^*h)=p^*(\Gamma_m(f,h)).
\end{equation}
Define the representative algebra of $M$ by
\[\mathcal R(M):=
\{f\in C^\infty(M,\C):p^*f\in\mathcal R(G)\},\]
and set $\mathcal R(M,\R):=\mathcal R(M)\cap C^\infty(M,\R)$. Equivalently, $\mathcal R(M)$ consists of the $G$-finite functions: those whose translates span a finite-dimensional vector space. Consider the left action $(\mathcal L_a f)(gK):=f(a^{-1}gK)$ and $\mathcal L_{ab}=\mathcal L_{a}\mathcal L_{b}$. Following \cite{van2005lie}, we would like to define functions that satisfy properties similar to those of matrix coefficients and are, in particular, well defined on the quotient. For $\alpha\in\Lambda_G^+$, define $(V_\alpha^*)^K$ as the space of linear forms on $V_\alpha$ fixed by $K$:
\[(V_\alpha^*)^K=\{\xi\in V_\alpha^*: \xi(\pi_\alpha(k)v)=\xi(v),\ \forall k\in K,\ \forall v\in V_\alpha\}.\]
For a fixed $\xi\neq 0$, each $v\in V_\alpha$ defines a function on $M$ by
\[f_{v,\xi}:M\to\C,\qquad f_{v,\xi}(gK):=\xi(\pi_\alpha(g^{-1})v).\]
The space $W_\xi:=\{f_{v,\xi}:v\in V_\alpha\}$ is a copy of $\pi_\alpha$: that is, the action of $G$ on $W_\xi$ is equivalent to $\pi_\alpha$. We define the isotypic component of highest weight $\alpha$\footnote{In \cite[p.~137]{Folland16}, this space is denoted by $M_{\pi_\alpha}$ and by $L^2(X)_\delta$ in \cite[p. 10]{van2005lie}} by
\[\mathcal R_\alpha(M)=\operatorname{span}_{\C}\{f_{v,\xi}:v\in V_\alpha,\ \xi\in(V_\alpha^*)^K\}, \qquad \mathcal R_\alpha(M)\simeq V_\alpha\otimes(V_\alpha^*)^K.\]
Here the tensor factor $(V_\alpha^*)^K$ records the multiplicity and Peter--Weyl theory gives
\[L^2(M)=\widehat{\bigoplus}_{\alpha\in\Lambda_G^+}\mathcal R_\alpha(M), \qquad \mathcal R(M)=\bigoplus_{\alpha\in\Lambda_G^+}^{\mathrm{alg}} \mathcal R_\alpha(M).\]
These statements follow from the Peter--Weyl decomposition on $G$ by taking right $K$-invariants; see \cite[Chapter~5]{Folland16}. In particular, $\mathcal R(M)$ is dense in $L^2(M)$. An irreducible representation and its dual have the same Casimir scalar. Thus, with the above convention for the left action, \eqref{eq:homogeneous-pullback} gives $-\Delta_m f=c_\alpha f$, $f\in\mathcal R_\alpha(M)$, where $c_\alpha$ is the scalar in \eqref{eq:coefficient-eigenvalue}. For each $\alpha$, let $\mathcal R_\alpha^+(M)$ be the space of highest-weight vectors in $\mathcal R_\alpha(M)$. In each irreducible copy of $V_\alpha$, these vectors form a distinguished one-dimensional subspace; their $G$-translates span that copy. The precise definition is given in \Cref{app:homogeneous-weights}. For the saturation proof, we use the algebra
\[\mathcal R^+(M):=\bigoplus_{\alpha\in\Lambda_G^+}^{\mathrm{alg}}\mathcal R_\alpha^+(M),\]
and the fact that multiplication adds its labels:
\[f\in\mathcal R_\alpha^+(M),\quad h\in\mathcal R_\beta^+(M)\quad\Longrightarrow\quad fh\in\mathcal R_{\alpha+\beta}^+(M).\]
Using the Laplacian product rule and \eqref{eq:parameter-gap}, we obtain
\begin{equation}
\label{eq:homogeneous-Cartan}
\Gamma_m(f,h)=-\langle\alpha,\beta\rangle_*fh, \qquad f\in\mathcal R_\alpha^+(M),\quad h\in\mathcal R_\beta^+(M).
\end{equation}
The following finiteness property follows from Hilbert's theorem and the Lie-algebra form of the Hochschild--Mostow theorem \cite[pp.~646--647]{HM73}. The proof is given in \Cref{app:homogeneous-weights}.

\begin{proposition}[Finite generation of highest-weight functions]
\label{prop:highest-weight-generators}
There exist finitely many nonconstant functions $q_1,\ldots,q_s$ and nonzero weights $\alpha_1,\ldots,\alpha_s\in\Lambda_G^+$ such that
\[q_j\in\mathcal R_{\alpha_j}^+(M), \qquad \mathcal R^+(M)=\C[q_1,\ldots,q_s].\]
\end{proposition}

\section{Finite saturation on compact Lie groups}
\label{sec:finite-saturation-groups}

The first point is to recover both signs of the quadratic terms generated by phase saturation. We give the argument under a transitive isometric action, so it applies both to $G$ and to $G/K$. For a finite-dimensional real space $\mathcal W\subset C^\infty(M,\R)$, define
\[\mathcal S_1(\mathcal W):=\mathcal W,\qquad
\mathcal S_{n+1}(\mathcal W):=\mathcal S_n(\mathcal W)+
\operatorname{span}_{\R}\{\Gamma(f,h):f,h\in\mathcal S_n(\mathcal W)\},\]
and set $\mathcal S_\infty(\mathcal W):=\bigcup_{n\geq1}\mathcal S_n(\mathcal W)$.

\begin{lemma}
\label{lem:recovery-signs}
Assume that a compact group acts transitively by isometries on $(M,g)$, and that $\mathcal W$ is invariant under this action and contains $1$. Then
\[\mathcal S_\infty(\mathcal W)
\subset\mathcal H_\infty(\mathcal W)
\cap\bigl(-\mathcal H_\infty(\mathcal W)\bigr).\]
\end{lemma}

\begin{proof}
Each $\mathcal H_n$ is stable under multiplication by nonnegative scalars. Moreover, induction gives
\[\mathcal H_n+\mathcal H_m\subset\mathcal H_{n+m-1},
\qquad n,m\geq1.\]
For $m=1$, this follows by induction on $n$ since $\mathcal H_1=\mathcal W$ is a vector space. For the induction step, write $k=k_0-\alpha B(f)\in\mathcal H_{m+1}$, with $k_0,f,-f\in\mathcal H_m$. If $h\in\mathcal H_n$, then $h+k_0\in\mathcal H_{n+m-1}$, and $f,-f$ belong to the same space. Thus $h+k\in\mathcal H_{n+m}$. Consequently, $\mathcal H_\infty$ is a convex cone, and
\[\mathcal V:=\mathcal H_\infty\cap(-\mathcal H_\infty)\]
is a real vector space. It is invariant under the action, since $B$ commutes with isometries. Let $E\subset\mathcal V$ be finite-dimensional and invariant, and let $e_1,\ldots,e_N$ be an orthonormal basis in $L^2(M,\R)$. The sum $\sum_j B(e_j)$ is independent of this basis and invariant under the action. By transitivity,
\begin{equation}
\label{eq:energy-balance}
\sum_{j=1}^N B(e_j)=C_E
\end{equation}
for a constant $C_E$. Since $e_j,-e_j\in\mathcal H_\infty$, we have $-B(e_j)\in\mathcal H_\infty$. Constants belong to $\mathcal V$, so
\[B(e_1)=C_E+\sum_{j=2}^N(-B(e_j))\in\mathcal H_\infty.\]
Choosing $e_1=f/\|f\|_{L^2}$ shows that $B(f)\in\mathcal V$ for every $f\in E$. Polarization gives $\Gamma(f,h)\in\mathcal V$ for $f,h\in E$. Each $\mathcal S_n(\mathcal W)$ is finite-dimensional and invariant, so induction proves the result.
\end{proof}

The following elementary lemma handles the pairs of weights with zero scalar product.

\begin{lemma}
\label{lem:finite-exceptional-set}
Let $w_1,\ldots,w_s$ be nonzero vectors in a finite-dimensional Euclidean space. For $n\in\N_0^s$, where $\N_0:=\N\cup\{0\}$, set $w(n):=\sum_j n_jw_j$. Then
\[\mathcal J:=\left\{
n\in\N_0^s\setminus\{0\}:
\langle w(n)-w_j,w_j\rangle=0\text{ whenever }n_j>0
\right\}\]
is finite. It contains every standard basis vector of $\N_0^s$.
\end{lemma}

\begin{proof}
Fix the support $I:=\{j:n_j>0\}$. For $n\in\mathcal J$, the vector $w(n)$ belongs to $E_I:=\operatorname{span}\{w_j:j\in I\}$ and satisfies $\langle w(n),w_j\rangle=\|w_j\|^2$ for $j\in I$. These equalities determine at most one vector $w_I\in E_I$, since the difference of two solutions belongs to $E_I$ and is orthogonal to $E_I$. For such a vector,
\[\|w_I\|^2=\sum_{j\in I}n_j\|w_j\|^2.\]
All coefficients are positive, so only finitely many integer tuples are possible. There are finitely many supports. The last point follows directly from the definition.
\end{proof}

Let $\mathcal D_G=\{\alpha_1,\ldots,\alpha_s\}$ be the explicit finite set of nonzero generators of $\Lambda_G^+$ constructed in \Cref{app:explicit-generators}. We apply \Cref{lem:finite-exceptional-set} with the scalar product $\langle\cdot,\cdot\rangle_*$, and denote the resulting set by $\mathcal J_G$. Write $\alpha(n):=\sum_j n_j\alpha_j$ and set $\mathcal F_G:=\{\alpha(n):n\in\mathcal J_G\}$. Define
\begin{equation}
\label{eq:control-space-G}
\mathcal W_G:=\operatorname{span}_{\R}
\left(
\{1\}\cup
\{\Re f,\Im f:f\in\mathcal E_\alpha,\ \alpha\in\mathcal F_G\}
\right).
\end{equation}
This space is finite-dimensional, invariant under $G\times G$, and contains the constants. Empty families give $\mathcal W_G=\R1$ when $G$ is trivial.

\begin{theorem}[Saturation on a compact Lie group]
\label{thm:finite-saturation-G}
Let $G$ be a compact connected Lie group endowed with a bi-invariant metric, and let $\mathcal W_G$ be defined by \eqref{eq:control-space-G}. Then
\[\mathcal R(G,\R)\subset\mathcal H_\infty(\mathcal W_G).\]
Consequently, $\mathcal H_\infty(\mathcal W_G)$ is dense in $L^2(G,\R)$.
\end{theorem}

\begin{proof}
Set $\mathcal A:=\mathcal S_\infty(\mathcal W_G)^\C$. It is invariant under $G\times G$ and stable under $\Gamma$. We prove $\mathcal E_{\alpha(n)}\subset\mathcal A$ by induction on $|n|:=\sum_j n_j$. The case $n=0$ is the space of constants, and the cases $n\in\mathcal J_G$ hold by construction. In particular, $\mathcal E_{\alpha_j}\subset\mathcal A$ for every $j$. If $n\neq0$ and $n\notin\mathcal J_G$, choose $j$ with $n_j>0$ and
\[\langle\alpha(n)-\alpha_j,\alpha_j\rangle_*\neq0.\]
The induction hypothesis gives $\mathcal E_{\alpha(n)-\alpha_j}\subset\mathcal A$. By \Cref{lem:Cartan-propagation}, $\mathcal E_{\alpha(n)}\subset\mathcal A$. Since $\mathcal D_G$ generates $\Lambda_G^+$, this proves $\mathcal R(G)\subset\mathcal A$. Taking real parts and applying \Cref{lem:recovery-signs} gives the inclusion. Density follows from Peter--Weyl.
\end{proof}

\begin{proof}[Proof of \Cref{thm:main-phases-G}]
Apply \Cref{thm:CP-phase} to the density statement of \Cref{thm:finite-saturation-G}.
\end{proof}
For $SO(3)$, the highest-weight labels are $\alpha_\ell=\ell u$, $\ell\in\N\cup\{0\}$, where $u$ is the positive root; see \Cref{app:weight-parameters}. Write $\mathcal E_\ell:=\mathcal E_{\pi_\ell}=\mathcal E_{\alpha_\ell}$, $c_\ell:=c_{\alpha_\ell}$, and $\nu:=\langle u,u\rangle_*>0$. Since $\rho=u/2$, the Casimir formula \eqref{eq:Casimir-eigenvalue} gives
\[c_\ell=\nu\ell(\ell+1),\qquad\langle\alpha_\ell,\alpha_j\rangle_*=\nu\ell j.\]
Taking $\mathcal D_{SO(3)}=\{u\}$, the exceptional condition becomes $(n-1)\nu=0$, so $\mathcal J_{SO(3)}=\{1\}$ and $\mathcal F_{SO(3)}=\{u\}$. Thus
\[\mathcal W_{SO(3)}=\operatorname{span}_{\R}\{1,R\mapsto R_{ab}:1\leq a,b\leq3\}.\]
The nine matrix entries span $\mathcal E_1$. Recovery of signs and \Cref{lem:Cartan-propagation} then generate $\mathcal E_2,\mathcal E_3,\ldots$, since $\langle\alpha_\ell,\alpha_1\rangle_*=\nu\ell>0$ for every $\ell\geq1$. Together with the constants, these spaces form the full representative algebra. Hence these ten potentials give phase control and, by \Cref{prop:main-global}, global $L^2$-STAC.

\begin{remark}
Here, we are not claiming that this number of controls is minimal. In fact, seven potentials suffice: $1$ and $R_{i1},R_{i2}$, $1\leq i\leq3$. Their span is invariant under left translations, so \Cref{lem:recovery-signs} applies. The identities $\Gamma(R_{i1},R_{j2})-\Gamma(R_{j1},R_{i2})=c^{-1}(R_{i1}R_{j2}-R_{j1}R_{i2})$ recover the third column, and hence generate the full coefficient space $\mathcal E_1$.
\end{remark}

\section{Finite saturation on normal homogeneous spaces}
\label{sec:finite-saturation-homogeneous}

We use the notation of \Cref{subsec:homogeneous-preliminaries}. On $G$, the left-right translates of any nonzero element of $\mathcal E_\alpha$ span the whole coefficient space. On $G/K$, the $G$-translates of a nonzero highest-weight vector span one irreducible copy, while an isotypic component may contain several such copies. To account for these multiplicities, we use generators of the full algebra $\mathcal R^+(M)$. Choose generators $q_1,\ldots,q_s$ as in \Cref{prop:highest-weight-generators}, with
highest weights $\alpha_1,\ldots,\alpha_s$.
Apply \Cref{lem:finite-exceptional-set} to these
weights with the scalar product $\langle\cdot,\cdot\rangle_*$,
and denote the resulting finite set by $\mathcal J$.
For $n\in\N_0^s$, write
\[q^n:=q_1^{n_1}\cdots q_s^{n_s},
\qquad
\alpha(n):=\sum_{j=1}^s n_j\alpha_j.\]
Define
\begin{equation}
\label{eq:control-space-homogeneous}
\mathcal W_M:=
\operatorname{span}_{\R}
\left(
\{1\}\cup
\left\{
\Re(\mathcal L_aq^n),\Im(\mathcal L_aq^n):
a\in G,\ n\in\mathcal J
\right\}
\right).
\end{equation}
This space is finite-dimensional because every $q^n$ is a representative function. It is $G$-invariant and contains $1$.

\begin{theorem}[Saturation on a normal homogeneous space]
\label{thm:finite-saturation-homogeneous}
For the space defined in \eqref{eq:control-space-homogeneous},
\[\mathcal R(M,\R)\subset\mathcal H_\infty(\mathcal W_M).\]
Consequently, $\mathcal H_\infty(\mathcal W_M)$ is dense in $L^2(M,\R)$.
\end{theorem}

\begin{proof}
Set $\mathcal A:=\mathcal S_\infty(\mathcal W_M)^\C$. This is a $G$-invariant complex vector space stable under $\Gamma_m$. We prove by induction on $|n|:=n_1+\cdots+n_s$ that $q^n\in\mathcal A$ for every $n\in\N_0^s$. The constant monomial belongs to $\mathcal A$. All monomials with $n\in\mathcal J$ belong to $\mathcal W_M^\C$ by construction. In particular, $q_j\in\mathcal A$, since each standard basis vector $e_j$ belongs to $\mathcal J$. If $n\neq0$ and $n\notin\mathcal J$, there exists $j$ with $n_j>0$ such that
\[\langle\alpha(n)-\alpha_j,\alpha_j\rangle_*\neq0.\]
By induction, $q^{n-e_j}\in\mathcal A$. Identity \eqref{eq:homogeneous-Cartan} gives
\[\Gamma_m(q^{n-e_j},q_j)=-\langle\alpha(n)-\alpha_j,\alpha_j\rangle_*q^n,\]
and hence $q^n\in\mathcal A$. Thus $\mathcal R^+(M)\subset\mathcal A$. Every irreducible subrepresentation of $\mathcal R(M)$ is generated by its highest-weight vector. Since $\mathcal A$ is $G$-invariant, it follows that $\mathcal R(M)\subset\mathcal A$. Taking real parts and applying \Cref{lem:recovery-signs} yields $\mathcal R(M,\R)\subset\mathcal H_\infty(\mathcal W_M)$. Density follows from Peter--Weyl theory.
\end{proof}

\begin{proof}[Proof of \Cref{thm:main-phases-homogeneous}]
Apply \Cref{thm:CP-phase} to the density statement
of \Cref{thm:finite-saturation-homogeneous}.
\end{proof}
\begin{remark}
The group case is included in the normal homogeneous setting
by taking $K=\{1\}$. One may also use the identification
\[(G\times G)/\Delta G \simeq G, \qquad [(a,b)]\longmapsto ab^{-1},\]
where $\Delta G:=\{(a,a):a\in G\}$.
The quotient metric induced by $2\overline m\oplus2\overline m$ corresponds exactly to
$\overline m$ on $G$. We present the group case separately to give a direct construction of $\mathcal W_G$ and to make the saturation argument easier to follow.
\end{remark}

\begin{remark}
The algebraic part of the saturation argument relies on finite generation of the highest-weight algebra, the interaction identity \eqref{eq:homogeneous-Cartan}, and invariance under the group action. After recovering both signs, these properties allow us to generate the representative algebra without requiring the configuration space itself to carry a group structure. This accounts for the
extension from compact Lie groups to normal homogeneous spaces.
\end{remark}
\section{From phase control to global controllability}
\label{sec:global-control-short}

Let $(M,g)$ be a smooth, connected, closed Riemannian manifold, with the potential and controls of \eqref{eq:Sch}. We adapt the phase-to-transport argument of \cite{BP25} and use the gradient generation result of \cite{KMP26}. Call a unitary operator $L^2$-STAR if its action on each initial state can be approximated, up to a constant phase, in arbitrarily small time; see \cite[Definition~5]{BP25}. This property is stable under finite compositions and strong limits whose limit is unitary \cite[Lemma~6]{BP25}. For $P\in\Diff_0(M)$ and a smooth real vector field $X$, set
\[\mathcal U_P\psi:=J_P^{1/2}(\psi\circ P),\qquad T_Xf:=Xf+\frac12(\operatorname{div}_gX)f,\quad f\in C^\infty(M,\C).\]
Here $J_P$ is the Jacobian relative to the Riemannian density $d\mu$. The operator $\mathcal U_P$ is unitary, and the characteristic formula gives $e^{sT_X}=\mathcal U_{\Phi_X^s}$, where $T_X$ denotes its skew-adjoint closure and $\Phi_X^s$ is the flow of $X$.

\begin{theorem}[STAC of gradient flows]
\label{thm:gradient-flows-short}
If \eqref{eq:Sch} satisfies STAC of phases, then $\mathcal U_{\Phi_{\nabla_g\varphi}^s}$ is $L^2$-STAR for every $\varphi\in C^\infty(M,\R)$ and $s\in\R$.
\end{theorem}

\begin{proof}
Following \cite[Theorem~16]{BP25}, fix $\varphi\in C^\infty(M,\R)$. Then for $\delta>0$ 
\[e^{i\frac{\varphi}{2\delta}} i\delta(\Delta_g-V)e^{-i\frac{\varphi}{2\delta}}+\frac{i}{4\delta}B(\varphi)=i\delta(\Delta_g-V)+T_{\nabla_g\varphi}, \qquad \text{ on } C^\infty(M,\C).\]
Let $C_\delta$ be the closure of this operator. It is skew-adjoint, since the left-hand side is a bounded skew-adjoint perturbation of a unitary conjugate of $i\delta(\Delta_g-V)$. The Trotter--Kato formula \cite[Theorem~VIII.31]{RS72} gives
\[\left(e^{iB(\varphi)/(4\delta n)} e^{i\frac{\varphi}{2\delta}}e^{i\delta(\Delta_g-V)/n}e^{-i\frac{\varphi}{2\delta}}\right)^n\longrightarrow e^{C_\delta}\quad(n\to\infty  \text{ strongly in }L^2).\]
The phase factors are $L^2$-STAR, and the free factors have total physical duration $\delta$. Hence $e^{C_\delta}$ can be approximated in time less than $\delta+\eta$ for every $\eta>0$. Moreover, $C_\delta f\to T_{\nabla_g\varphi}f$ on the common core $C^\infty(M,\C)$, so $e^{C_\delta}\to e^{T_{\nabla_g\varphi}}$ strongly by \cite[Theorems~VIII.21 and~VIII.25(a)]{RS72}. Letting $\delta\downarrow0$ and replacing $\varphi$ by $s\varphi$ proves the claim.
\end{proof}

\begin{corollary}[STAC of diffeomorphisms]
\label{cor:flows-diffeomorphisms-short}
If \eqref{eq:Sch} satisfies STAC of phases, then $\mathcal U_P$ is $L^2$-STAR for every $P\in\Diff_0(M)$.
\end{corollary}

\begin{proof}
By \cite[Corollary~1]{KMP26}, compositions of gradient flows are dense in $\Diff_0(M)$. The flow product formulas in its proof give approximations $P_n\to P$ in $C^1$. By \Cref{thm:gradient-flows-short} and stability under composition, each $\mathcal U_{P_n}$ is $L^2$-STAR. Convergence of the maps and their Jacobians implies $\mathcal U_{P_n}\psi\to\mathcal U_P\psi$ for smooth $\psi$; unitarity extends this to all of $L^2$. Stability under strong limits concludes the proof.
\end{proof}

\begin{proof}[Proof of \Cref{prop:main-global}]
We conclude as in \cite[Theorem~12]{BP25}. Normalized states $q=a e^{i\varphi}$ with $a>0$ and $a,\varphi\in C^\infty(M,\R)$ are dense in $\S_M$, by smooth approximation of the modulus and a measurable argument. For two such states $q_j=a_je^{i\varphi_j}$, Moser's transport argument for positive densities \cite{Moser65} gives $P\in\Diff_0(M)$ satisfying
\[P^*(a_0^2d\mu)=a_1^2d\mu,\qquad e^{i(\varphi_1-\varphi_0\circ P)}\mathcal U_Pq_0=q_1.\]
The operator in the second identity is $L^2$-STAR by \Cref{cor:flows-diffeomorphisms-short} and phase control. Density and unitarity then give approximate controllability on $\S_M$ up to a constant phase, in arbitrarily small positive time; positivity of the time follows from strong continuity of the free evolution. Write $1=\sum_{j=1}^m a_jW_j$, with $a_j\in\R$. For a trajectory of duration $T>0$ and final phase $e^{i\theta}$, replacing $u_j$ by $u_j+\theta a_j/T$ multiplies the final state by $e^{-i\theta}$. This removes the constant phase without changing $T$ and yields $L^2$-STAC on $\S_M$. The density formulation of Moser's argument also applies when $M$ is nonorientable.
\end{proof}

\begin{remark}
\label{rem:alternative-density-transport-short}
Alternatively, \cite[Proposition~3.1 and Remark~3.2]{KL09} transports a positive density path by time-dependent gradients. For $\rho_t=(1-t)a_0^2+ta_1^2$, solve
\[\operatorname{div}_g(\rho_t\nabla_g f_t)=-\partial_t\rho_t,\qquad\int_M f_t\,d\mu=0.\]
Positivity and the zero-mean right-hand side give a smooth solution. Its gradient flow transports the endpoint densities; freezing the gradients approximates this flow in $C^1$ by compositions of autonomous gradient flows. Thus the preceding phase correction gives global control directly from \Cref{thm:gradient-flows-short}.
\end{remark}

\section*{Acknowledgements}
We would like to thank Felipe Labra for taking the time to read this manuscript and for his helpful comments.
\appendix
\crefalias{section}{appendix}
\crefalias{subsection}{subappendix}
\crefalias{subsubsection}{subsubappendix}
\section{Representation-theoretic details}
\label{app:representation-details}

For the reader's convenience, this appendix recalls the representation-theoretic background, fixes the normalizations, gives the explicit generating family used in \Cref{sec:finite-saturation-groups}, and proves \Cref{prop:highest-weight-generators}.

\subsection{Structure of compact Lie groups}
\label{app:weight-parameters}

We now use the structure theorem for compact connected Lie groups; see \cite[pp.~161--162]{Folland16} and \cite[Chapter~V, Theorem~8.1 and Exercise~8.7(6)]{BTD85}.

\begin{proposition}
\label{thm:structure}
Let $G$ be a compact connected Lie group. Then there exist nonnegative integers $d,r$ and a finite central covering homomorphism
\begin{equation}
\label{eq:central-cover}
\Pi:\widetilde G:=\T^d\times G_1\times\cdots\times G_r\longrightarrow G,
\end{equation}
where each $G_i$ is a compact, simply connected Lie group with simple Lie algebra. The kernel $F:=\ker\Pi$ is a finite subgroup of the center of $\widetilde G$, and $G\cong\widetilde G/F$.
\end{proposition}

On $\T^d=\R^d/(2\pi\Z)^d$, the irreducible unitary representations are the characters $\chi_k([x])=e^{i\langle k,x\rangle}$, $k\in\Z^d$. For each simple factor $G_i$, fix a maximal torus $T_i$ and write $\mathfrak g_i:=\Lie(G_i)$, $\mathfrak t_i:=\Lie(T_i)$ and $\ell_i:=\dim T_i$. Its integral weight lattice is
\[P_i:=\{\lambda\in\mathfrak t_i^*: \lambda(H)\in2\pi\Z\text{ whenever }\exp_{T_i}(H)=1\}.\]
It parametrizes the characters $\chi_\lambda:T_i\to\S^1$ by $\chi_\lambda(\exp_{T_i}H)=e^{i\lambda(H)}$; see \cite[Chapter~II, Proposition~8.1]{BTD85} and \cite[Proposition~4.58]{Knapp02}. We identify roots with real linear forms on $\mathfrak t_i$ through the convention
$[H,X_\gamma]=i\gamma(H)X_\gamma$. Fix positive roots $\Phi_i^+$ for $\mathfrak g_i\otimes_\R\C$ relative to $\mathfrak t_i\otimes_\R\C$, and let $\omega_{i,1},\ldots,\omega_{i,\ell_i}$ be the corresponding fundamental weights. Since $G_i$ is simply connected, they form a $\Z$-basis of $P_i$, and the dominant integral weights are
\[P_i^+=\left\{\sum_{j=1}^{\ell_i}n_j\omega_{i,j}:n_j\in\N_0\right\};\]
see \cite[Proposition~4.59 and Theorem~5.107]{Knapp02}. For $\lambda_i\in P_i^+$, let $\pi_{i,\lambda_i}:G_i\to U(V_{i,\lambda_i})$ be the irreducible representation with highest weight $\lambda_i$. The highest-weight theorem \cite[Theorem~5.110]{Knapp02} and the representation theory of products \cite[Theorem~7.17]{Folland16} parametrize the irreducible representations of $\widetilde G$ by
\[\alpha=(k,\lambda_1,\ldots,\lambda_r)\in\Z^d\times P_1^+\times\cdots\times P_r^+.\]
The corresponding representation is
\[\widetilde\pi_{k,\lambda_1,\ldots,\lambda_r}(g_0,g_1,\ldots,g_r)=\chi_k(g_0)\pi_{1,\lambda_1}(g_1)\otimes\cdots\otimes\pi_{r,\lambda_r}(g_r),\]
with the empty tensor product interpreted as $\C$. It descends to $G$ if and only if it is trivial on $F$. Thus the admissible parameters are
\[\Lambda_G^+:=\left\{(k,\lambda_1,\ldots,\lambda_r)\in\Z^d\times P_1^+\times\cdots\times P_r^+:\widetilde\pi_{k,\lambda_1,\ldots,\lambda_r}|_F=\id\right\}.\]
For $\alpha\in\Lambda_G^+$, write $\pi_\alpha$ for the corresponding representation of $G$, $V_\alpha$ for its representation space, and $\mathcal E_\alpha:=\mathcal E_{\pi_\alpha}$.
\begin{remark}
The highest-weight theorem also parametrizes $\widehat G$
directly by the dominant characters of a maximal torus of $G$;
see \cite[Theorem~5.110]{Knapp02}.
The description through the central covering
\eqref{eq:central-cover} provides concrete coordinates for
the Casimir formula in \Cref{app:Casimir-details}
and the construction of a finite generating family
in \Cref{app:explicit-generators}.
\end{remark}
For $SO(3)$, the covering is $\Pi:SU(2)\to SO(3)$ with kernel $\{\pm I\}$. In $SU(2)$, take $T=\{t_\theta:=\operatorname{diag}(e^{i\theta},e^{-i\theta}):\theta\in\R\}$ and $\chi_\omega(t_\theta)=e^{i\theta}$. Then $\omega$ is the fundamental weight and $u:=2\omega$ the positive root. The irreducible representation $\rho_n$ of highest weight $n\omega$ has dimension $n+1$ and satisfies
\[\rho_n(-I)=(-1)^n\operatorname{Id};\]
see \cite[Proposition~4.35]{Hall15}. Hence it descends to $SO(3)$ exactly when $n$ is even. Writing $\pi_\ell\circ\Pi=\rho_{2\ell}$, we obtain
\[\Lambda_{SO(3)}^+=2\N_0\omega=\N_0u,\qquad \dim_\C V_{\pi_\ell}=2\ell+1,\]
and $\mathcal E_{2\ell\omega}=\mathcal E_{\pi_\ell}$.

\begin{remark}
\label{rem:quaternion-cover}
Realize $SU(2)\simeq\operatorname{Spin}(3)$ as the group of unit quaternions in $\mathbb H$, with $|a+bi+cj+dk|^2=a^2+b^2+c^2+d^2$; its underlying manifold is $S^3$. See \cite[Chapter~I, Sections~8 and~17]{Knapp02}. Identifying $\R^3$ with the purely imaginary quaternions gives the covering $\Pi(q)(v)=qvq^{-1}$, $v\in\operatorname{Im}\mathbb H\simeq\R^3$. Its kernel is $\{\pm1\}$, corresponding to $\{\pm I\}$ in $SU(2)$, and $S^3$ is simply connected, so this is the universal covering.
\end{remark}

\subsection{The Casimir and the Cartan component}
\label{app:Casimir-details}

The pullback metric $\widetilde m:=\Pi^*\overline m$ makes $\Pi$ a local isometry. At the identity, the tangent space is $T_1\widetilde G=\R^d\oplus\mathfrak g_1\oplus\cdots\oplus\mathfrak g_r$. Bi-invariance implies that the inner product $b:=\widetilde m_1$ decomposes as the orthogonal sum
\[b=b_0\oplus(-s_1\kappa_1)\oplus\cdots\oplus(-s_r\kappa_r),
\qquad s_i>0,\]
where $b_0$ is an inner product on $\R^d$ and $\kappa_i$ is the Killing form of $\mathfrak g_i$; see \cite[Theorem~2.45 and Proposition~2.48]{AB15}. Let $Q$ be the inner product on the toral weight space dual to $b_0$. Thus, if $b_0(u,v)=u^\top Mv$, then $Q(k,\ell)=k^\top M^{-1}\ell$. On each $\mathfrak t_i^*$, use the inner product $\langle\cdot,\cdot\rangle_i$ dual to $-\kappa_i$, and set $a_i:=s_i^{-1}$, $\rho_i:=\tfrac12\sum_{\gamma\in\Phi_i^+}\gamma$. The Casimir eigenvalue formula \cite[Proposition~5.28(b)]{Knapp02} gives $-\Delta_Gf=c_\alpha f$ for $f\in\mathcal E_\alpha$, where
\begin{equation}
\label{eq:Casimir-eigenvalue}
c_\alpha=Q(k,k)+\sum_{i=1}^r a_i
\left(\|\lambda_i+\rho_i\|_i^2-\|\rho_i\|_i^2\right),
\qquad \alpha=(k,\lambda_1,\ldots,\lambda_r).
\end{equation}
For $\alpha=(k,\lambda_1,\ldots,\lambda_r)$ and $\beta=(\ell,\mu_1,\ldots,\mu_r)$, write
\[\langle\alpha,\beta\rangle_*:=
Q(k,\ell)+\sum_{i=1}^r a_i\langle\lambda_i,\mu_i\rangle_i.\]
Choose highest-weight vectors $v_\alpha\in V_\alpha$. The tensor $v_\alpha\otimes v_\beta$ generates the irreducible component of highest weight $\alpha+\beta$, with multiplicity one, called the Cartan component; see \cite[p.~309]{Baur03}. It is trivial on $F$ when both factors are trivial on $F$, so $\alpha+\beta\in\Lambda_G^+$. Expanding \eqref{eq:Casimir-eigenvalue} gives
\begin{equation}
\label{eq:Casimir-gap}
\frac{c_\alpha+c_\beta-c_{\alpha+\beta}}2
=-\langle\alpha,\beta\rangle_*.
\end{equation}
Choose $\xi_\alpha\in V_\alpha^*$ with $\xi_\alpha(v_\alpha)=1$, and set $q_\alpha(g):=\xi_\alpha(\pi_\alpha(g)v_\alpha)$. Since $v_\alpha\otimes v_\beta$ belongs to the Cartan component $V_{\alpha+\beta}$, the product $q_\alpha q_\beta$ is a coefficient of this component, obtained by restricting $\xi_\alpha\otimes\xi_\beta$ to it. Its value at the identity is one. Hence
\begin{equation}
\label{eq:Cartan-product}
q_\alpha q_\beta\in\mathcal E_{\alpha+\beta}\setminus\{0\}.
\end{equation}
\subsection{Explicit generators of the representation parameters}
\label{app:explicit-generators}

This construction specifies the set $\mathcal D_G$ used in \Cref{sec:finite-saturation-groups}. We construct a finite generating family for $\Lambda_G^+$ using the central covering \eqref{eq:central-cover}. The finite generation of $\Lambda_G^+$ also follows from Gordan's lemma; see \cite[Theorem~2.4(ii)]{Brion10}. Let $N$ be the exponent of $F$, with $N=1$ when $F=\{1\}$. Let $e_1,\ldots,e_d$ be the standard basis of the toral character lattice, and let $\boldsymbol\omega_{i,j}$ denote the weight equal to $\omega_{i,j}$ on the $i$th simple factor and zero elsewhere. By Schur's lemma, central elements act by scalars, and these scalars multiply when highest weights are added. Since $f^N=1$ for every $f\in F$, the weights $(\pm Ne_\nu,0,\ldots,0)$ and $N\boldsymbol\omega_{i,j}$ belong to $\Lambda_G^+$. Define the admissible residues by
\[\mathcal Q_N:=\left\{
\varrho=(r_0,r_1,\ldots,r_r)\in\Lambda_G^+:
\begin{array}{l}
r_0=\displaystyle\sum_{\nu=1}^d a_\nu e_\nu,\quad
r_i=\displaystyle\sum_{j=1}^{\ell_i}a_{i,j}\omega_{i,j},\\
a_\nu,a_{i,j}\in\{0,\ldots,N-1\}
\end{array}
\right\},\]
and set
\[\mathcal D_G:=
\{(\pm Ne_\nu,0,\ldots,0):1\leq\nu\leq d\}
\cup\{N\boldsymbol\omega_{i,j}:1\leq i\leq r,\ 1\leq j\leq\ell_i\}
\cup(\mathcal Q_N\setminus\{0\}).\]
These are nonzero generators of $\Lambda_G^+$. Indeed, division of the toral coordinates and the fundamental-weight coefficients by $N$ writes any $\alpha\in\Lambda_G^+$ as
\[\alpha=\varrho+(Nq,0,\ldots,0)
+N\sum_{i,j}m_{i,j}\boldsymbol\omega_{i,j},
\qquad q\in\Z^d,\quad m_{i,j}\in\N_0.\]
The remainder $\varrho$ is admissible because the central character of the subtracted part is trivial on $F$. The signed toral generators account for every $q\in\Z^d$.  For $SO(3)$, the kernel has exponent $N=2$, there is no toral factor, and $\mathcal Q_N=\{0\}$. Hence $\mathcal D_{SO(3)}=\{2\omega\}$, with $u=2\omega$ in the example of \Cref{sec:finite-saturation-groups}. That example gives $\mathcal J_{SO(3)}=\{1\}$ and $\mathcal F_{SO(3)}=\{2\omega\}$, yielding the ten-potential family described there.

\subsection{Highest-weight functions on normal homogeneous spaces}
\label{app:homogeneous-weights}
Let $\mathfrak n_+$ be the sum of the chosen positive root spaces in $\mathfrak g_\C:=\mathfrak g\otimes_\R\C$. In the notation of \Cref{subsec:homogeneous-preliminaries}, the highest-weight space is
\[\mathcal R_\alpha^+(M):=\{f\in\mathcal R_\alpha(M):d\mathcal L(X)f=0\text{ for every }X\in\mathfrak n_+\}.\]
Infinitesimal translations act as derivations, so multiplication adds highest weights. The Laplacian product rule and \eqref{eq:Casimir-gap} then give \eqref{eq:homogeneous-Cartan}.

\begin{proof}[Proof of \Cref{prop:highest-weight-generators}]
Let $G_\C$ be the algebraic complexification of $G$, and let $K_\C$ be the Zariski closure of $K$ in $G_\C$; the latter is reductive. Restriction identifies $\C[G_\C]$ with $\mathcal R(G)$, compatibly with right translations by $K$; see \cite[Chapter~III, Section~8, Proposition~8.2 and the preceding paragraph]{BTD85}. For $f\in\C[G_\C]^K$ and fixed $g\in G_\C$, the regular function $h\mapsto f(gh)-f(g)$ vanishes on $K$, hence on $K_\C$ by Zariski density. Thus
\[\mathcal R(M)\simeq\C[G_\C]^K=\C[G_\C]^{K_\C}.\]
Hilbert's theorem for reductive invariants makes this algebra finitely generated; see \cite[p.~646]{HM73}. The semisimple Lie algebra $[\mathfrak g_\C,\mathfrak g_\C]$ acts on it locally finitely by derivations. The Lie-algebra form of the Hochschild--Mostow theorem \cite[p.~647]{HM73} gives finite generation of the subalgebra annihilated by its positive root spaces $\mathfrak n_+$, which is precisely $\mathcal R^+(M)$ \footnote{The finite generation of $\mathcal R^+(M)$ also follows from
the Hadziev--Grosshans theorem; see \cite[Theorem~2.7]{Brion10}}. For a torus, $\mathfrak n_+=0$ and $\mathcal R^+(M)=\mathcal R(M)$. Taking the finitely many weight components of each generator with respect to a maximal torus of $G$ gives generators of definite highest weight. Weight-zero generators are $G$-invariant, hence constant by transitivity, and may be discarded.
\end{proof}

\bibliographystyle{siam}
\bibliography{Bilinear_Lie}

\end{document}